\documentclass[a4paper,11pt]{article}

\usepackage[top=3.0cm, bottom=3.0cm, inner=2.5cm, outer=2.5cm,
includefoot]{geometry}

\usepackage{verbatim}
\usepackage{caption}
\usepackage{url}
\usepackage{amsmath}
\usepackage{geometry}
\usepackage{amssymb}
\usepackage{amsmath}
\usepackage{graphicx}
\usepackage{amsthm}
\usepackage{bbm}
\usepackage{float}
\usepackage{color,soul}
\usepackage{hyperref}
\usepackage[T1]{fontenc}
\usepackage{authblk}
\usepackage{booktabs}
\usepackage{longtable}
\usepackage{enumerate}
\usepackage{tikz}
\usetikzlibrary{arrows}
\usetikzlibrary{decorations.markings}
\usepackage{indentfirst}
\newcommand{\eChar}{\begin{enumerate}[(i)]}
\newcommand{\eCharR}{\begin{enumerate}[(a)]}
\newcommand{\eBr}{\begin{enumerate}[(1)]}

\title
{
Edge-connectivity of graphs with non-negative Bakry-\'Emery curvature and amply regular graphs
}

\author[1]{Kaizhe Chen\thanks{Email: ckz22000259@mail.ustc.edu.cn}}
\author[2]{Jack H. Koolen\thanks{Email: koolen@ustc.edu.cn}}
\author[3]{Shiping Liu\thanks{Email: spliu@ustc.edu.cn}}
\affil[1]{School of the Gifted Young, University of Science and Technology of China}
\affil[2,3]{School of Mathematical Sciences, University of Science and Technology of China}

\date{}

\theoremstyle{plain}
\newtheorem{lemma}{Lemma}[section]
\newtheorem{theorem}[lemma]{Theorem}
\newtheorem{proposition}[lemma]{Proposition}
\newtheorem{corollary}[lemma]{Corollary}

\theoremstyle{definition}

\newtheorem{claim}{Claim}

\newtheorem{definition}[lemma]{Definition}

\newtheorem{remark}[lemma]{Remark}

\numberwithin{equation}{section}

\begin{document}

\maketitle

\begin{abstract}
We establish a sharp edge-connectivity estimate for graphs with non-negative Bakry-\'Emery curvature. This leads to a geometric criterion for the existence of a perfect matching.
Precisely, we show that any regular graph with non-negative Bakry-\'Emery curvature and an even or infinite number of vertices has a perfect matching. Through a synthesis of combinatorial and curvature-related techniques, we determine the edge-connectivity of (possibly infinite) amply regular graphs.
\end{abstract}

\section{Introduction}
The discrete Ricci curvature bridges continuous and discrete worlds, allowing powerful tools from Riemannian geometry to be applied in discrete settings (see \cite{NR17} and references therein). Bakry-\'Emery curvature \cite{BE} is a form of discrete Ricci curvature based on the Bochner identity, which has led to many fruitful results (see e.g. \cite{CHLZ,HPS24,HL17,LY10,LMP24,SalezGAFA,SalezJEMS}).

A {\it{matching}} of a graph $G$ is a set of mutually disjoint edges. A matching is {\it{perfect}} if every vertex of $G$ is incident with an edge of the matching.
Determining the existence of perfect matchings in regular graphs is a fundamental problem in graph theory and combinatorial optimization with profound theoretical and algorithmic implications.
%A classical result states that a connected claw-free graph with even number of vertices has a perfect matching \cite{LasV75, Sumner74}. 
For a $d$-regular graph $G$ with an even number of vertices, a well-known result of Brouwer and Haemers \cite[Theorem 3.1]{BrHa05} shows that if the third largest adjacency eigenvalue of $G$ is at most $k-1+\frac{3}{k+2}$, then $G$ has a perfect matching. This condition was further sharpened in \cite{CGH09}. In this paper, we obtain a geometric condition for (possibly infinite) regular graphs to have perfect matchings.

\begin{theorem}\label{BEPM}
    Let $G$ be a connected regular graph with an even or infinite number of vertices. If $G$ has non-negative Bakry-\'Emery curvature, then $G$ has a perfect matching. 
\end{theorem}

%IF $G$ is $d$ regular with $d$ odd and has non-negative Bakry-\'Emery curvature, then $G$ has a perfect matching. 

Note that Riemannian manifolds (or graphs) with non-negative curvature can be non-compact (or infinite), whereas those with uniform positive curvature lower bounds must be compact (or finite). This constitutes a fundamental distinction between non-negative and positive curvature conditions. 
For instance, the finite hypercube graphs have positive Bakry-\'Emery curvature, while the infinite path has zero Bakry-\'Emery curvature.
%We also mention that any hypercube graph has positive Bakry-\'Emery curvature and is not claw-free when its dimension is at least $3$.
 
A graph $G=(V,E)$ of at least two vertices is called {\it $l$-edge-connected} if $G-F$ is connected for every set $F \subset E$ of fewer than $l$ edges.
The {\it edge-connectivity} of $G$ is the greatest integer $l$ such that $G$ is $l$-edge-connected. 
The edge-connectivity of a single vertex is defined to be 0.
%An edge set $F$ in a connected graph $G$ is called {\it{disconnecting}} if $G-F$ is disconnected.
The following classic theorem (see e.g. \cite{CGS}) shows that a regular graph with large edge-connectivity must contain a perfect matching.

\begin{theorem}[{\cite{CGS}}]\label{PM}
    Let $G$ be a connected $d$-regular graph with an even number of vertices. If $G$ is $(d-1)$-edge-connected, then $G$ has a perfect matching.
\end{theorem}

The proof of Theorem \ref{PM} is based on a celebrated characterization for finite graphs possessing perfect matchings given by Tutte \cite{T47}, who later proved \cite{T52} that his criterion holds also for locally finite graphs. 
Thus, Theorem \ref{PM} also applies to regular graphs with infinitely many vertices.

We achieve Theorem \ref{BEPM} by proving the following stronger conclusion relating the curvature and edge-connectivity of graphs. According to Theorem \ref{PM}, Theorem \ref{BEPM} is a direct consequence of Theorem \ref{BE}.

\begin{theorem}\label{BE}
    Let $G$ be a connected graph with minimum degree $\delta$. If $G$ has non-negative Bakry-\'Emery curvature, then $G$ is $(\delta-1)$-edge-connected.
\end{theorem}

Theorem \ref{BE} is sharp for many graphs. For instance, let $G$ be the Cartesian product of the infinite path and the complete graph $K_k$ of size $k$. Then $G$ is $(k+1)$-regular with Bakry-\'Emery curvature zero and edge-connectivity $k$. However, all known graphs for which Theorem \ref{BE} is sharp are infinite. We conjecture that, for finite graphs, the conclusion of Theorem \ref{BE} can be improved to $\delta$-edge-connected.

We briefly review some previous works on relations between connectivity and discrete Ricci curvature. For  a connected graph $G$ with Bakry-\'Emery curvature $K$ and minimum degree $\delta$, Horn, Purcilly, and Stevens \cite{HPS24} prove that the connectivity of $G$ is at least $(2K+\delta +5)/8$. Lin-Lu-Yau curvature \cite{LLY,O} is another form of discrete Ricci curvature. The first and third name authors and You \cite{CLY} prove that the connectivity of a graph is at least the product of its Lin-Lu-Yau curvature lower bound and its minimum degree, and the edge-connectivity of a graph with positive Lin-Lu-Yau curvature is equal to its minimum degree.

%We extend a lemma of Horn. and make good use of a free constant.

%Amply regular graphs. Terwilliger's conjecture. Hopefully positive curvature. NOT YET confirmed. However, we can figure out the edge-connectivity of ARG in a strong way. 

The edge-connectivity of highly symmetric graphs has been widely studied. It has been shown that the edge-connectivity equals the degree for finite vertex-transitive graphs \cite{Mad71} and distance-regular graphs \cite{BrHa05,BK}. In this paper, we determine the edge-connectivity of (possibly infinite) amply regular graphs (see Section \ref{2} for the definition). Notice that all distance-regular graphs are amply regular.

\begin{theorem}\label{ARG}
    Let $G$ be a connected amply regular graph with parameters $(d,\alpha,\beta)$ with $\beta>1$. Then $G$ is $d$-edge-connected. Further, if $G$ is not a quadrangle, then the only disconnecting sets of $d$ edges are the sets of $d$ edges on a single vertex.
\end{theorem}

%This extends the previous line of works: Brouwer....

The above theorem does not hold for amply regular graphs with $\beta =1$. 
The infinite path serves as a counterexample; a finite counterexample is provided in Figure \ref{ARGexample}.

\begin{figure}[htbp!]
\centering
\includegraphics[scale=0.45]{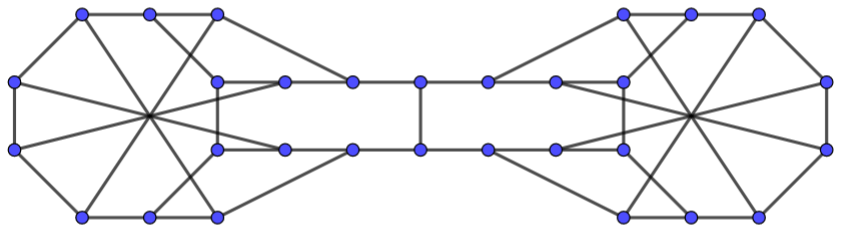}
\caption{An amply regular graph with degree 3 and edge-connectivity 2}
\label{ARGexample}
\end{figure}

It is worth noting that we know of no amply regular graph with $\beta >1$ which has non-positive Bakry-\'Emery curvature; nevertheless, it has not yet been established that every such graph has positive Bakry-\'Emery curvature. However, we prove Theorem \ref{ARG} through a synthesis of combinatorial and geometric techniques.

Finally, by combining Theorem \ref{PM} and Theorem \ref{ARG}, we obtain the following result on the existence of perfect matchings in amply regular graphs, which substantially generalizes the corresponding results for finite distance-regular graphs \cite[Corollary 4.2]{BrHa05}.

%\begin{corollary}
%Let $G$ be a connected regular graph of an even number of vertices. If $G$ satisfies $CD(\infty,0)$, then $G$ has a perfect matching.
%\end{corollary}

\begin{corollary}
     Let $G$ be a connected amply regular graph with parameters $(d,\alpha,\beta)$ with $\beta>1$. If $G$ has an even or infinite number of vertices, then $G$ has a perfect matching.
\end{corollary}

Throughout the paper, we use the following notations.  Let $G=(V,E)$ be a simple locally finite graph, and $V,E$ respectively denote the set of vertices and edges of $G$. For a vertex $x$ in $G$, we write $N_1(x)$ for the set of vertices adjacent to $x$ and $N_2(x)$ for the set of vertices at distance two from $x$, while the degree of $x$ is denoted by $d(x):= |N_1(x)|$.
For fixed vertex $x$, we simply write $N_1$ instead of $N_1(x)$ and $N_2$ instead of $N_2(x)$. 
For a vertex set $X$, let $d_X (x)$ denote the number of neighbors of $x$ in $X$.
For two vertices $x$ and $y$, we write $x\sim y$ if $x$ is adjacent to $y$.
For two disjoint vertex sets $X$ and $Y$, we write $e(X,Y)$ for the numbers of edges with one end in $X$ and the other in $Y$.

\section{Preliminaries}\label{2}
We first introduce the concept of Bakry-\'Emery curvature.
Let $G=(V,E)$ be a locally finite graph. The (standard) Laplacian $\Delta$ is defined as follows. For any function $f:V\to \mathbb{R}$ and any vertex $x\in V$, 
$$\Delta f(x):=\sum_{y\sim x}(f(y)-f(x)).$$
The gradient operator $\Gamma$ and the iterated gradient operator $\Gamma_2$ are defined as follows. For any two functions $f,g:V\to \mathbb{R}$ and any vertex $x\in V$,
\begin{align*}
\Gamma(f,g)(x):=&\frac{1}{2}\left(\Delta(fg)-(\Delta f)g-f(\Delta g)\right)(x).\\
\Gamma_{2}(f,g)(x):=&\frac{1}{2}\left(\Delta\Gamma(f,g)-\Gamma(\Delta f,g)-\Gamma(f,\Delta g)\right)(x).
\end{align*}
For simplicity, we denote $\Gamma(f)(x):=\Gamma(f,f)(x)$ and $\Gamma_{2}(f)(x):=\Gamma_{2}(f,f)(x)$.

\begin{definition}[Bakry--\'Emery curvature \cite{BE,LY10}] Let $G$ be a locally finite graph. For $K\in \mathbb{R}$ and $N\in (0,\infty]$, we say that $G$ satisfies the Bakry--\'Emery curvature dimension inequality $CD(N,K)$ at a vertex $x$, if for any function $f:V\rightarrow {\mathbb{R}}$, we have
\begin{equation}\label{curvatureineq}
    \Gamma_{2}(f)(x)\geq\frac{1}{N}(\Delta f)^{2}(x)+K\Gamma(f)(x).
\end{equation}
\end{definition}
Here, we use the notation $\frac{1}{N}=0$ when $N=\infty$.
We say that $G$ satisfies $CD(N,K)$, if it satisfies $CD(N,K)$ at every vertex.
We define the Bakry--\'Emery curvature of $G$ at $x$ as
$$K_{BE}(G,x):=\sup\{K: G\ \text{satisfies}\ CD(\infty,K)\ \text{at}\ x\}.$$
The Bakry--\'Emery curvature $K_{BE}(G)$ of $G$ is defined as the infimum of $K_{BE}(G,x)$ over all vertices $x$ in $G$.

We mention that the Bakry-\'Emery curvature $K_{BE}(G,x)$ of $G$ at $x$ can be reformulated as the minimal eigenvalue of the so-called curvature matrix at $x$, which is constructed from the local structure of $x$, see \cite{CKLP22,Si21}.

Horn, Purcilly, and Stevens \cite[Proposition 1]{HPS24} recently derive the following proposition of Bakry--\'Emery curvature, which is a basis of our later proof.

\begin{proposition}[\cite{HPS24}]\label{PH}
    Let $G=(V,E)$ be a graph and $x$ be a vertex in $G$. Then $G$ satisfies $CD(\infty, K)$ at $x$ if and only if for all functions $f:V\rightarrow \mathbb{R}$,
    \begin{align}\notag
        \sum_{z\in N_2}\sum_{\substack{y\in N_1 \\ y\sim z}}\left[ \frac{1}{4}(f(z)-f(y))^2-\frac{1}{2} (f(z)-f(y))(f(y)-f(x)) \right] +\sum_{\substack{yy'\in E\\ y,y'\in N_1}} (f(y)-f(y'))^2
        \\ \notag \ge \left( \frac{2K+d(x)-3}{2} \right)\Gamma (f)(x) -\frac{1}{2}(\Delta f(x))^2.
    \end{align}
\end{proposition}

We next recall the definition of an amply regular graph.

\begin{definition}[Amply regular graph \cite{BCN89}] 
    Let $G$ be a $d$-regular graph, which is neither complete nor empty. Then $G$ is called an amply regular graph with parameters $(d,\alpha,\beta)$ if any two adjacent vertices have exactly $\alpha$ common neighbors, and any two vertices at distance two have exactly $\beta$ common neighbors.
\end{definition}

The following theorem of amply regular graphs due to Brouwer and Neumaier \cite{BN} (see also \cite[Chapter 1.2]{BCN89}) is useful in the proof of Theorem \ref{ARG}.

\begin{theorem}[\cite{BN}]\label{BCN}
    An amply regular graph with parameters $(d,\alpha,2)$ such that $d < \frac{1}{2}\alpha(\alpha+ 3)$ contains no $K_{2,1,1}$.
\end{theorem}

For a fixed graph $G$, we simply write $K_{BE}(x)$ instead of $K_{BE}(G,x)$. Now, let us recall the Bakry--\'Emery curvature of amply regular graphs, see \cite[Theorem 4.3]{CHLZ}.

\begin{theorem}[\cite{CHLZ}]\label{thm:ARGCurvature}
    Let $G=(V,E)$ be an amply regular graph with parameters $(d,\alpha,\beta)$. Let $x\in V$ be a vertex and $A_{N_1}$ be the adjacency matrix of the local graph induced by $N_1(x)$. Then, 
    \begin{equation}\notag
    K_{BE}(x)=2+\frac{\alpha}{2}+\left(\frac{2d(\beta-2)-\alpha^2}{2\beta}+\frac{2}{\beta}\min_{\lambda\in \mathrm{sp} \left(A_{N_1}\right)}\left(\lambda-\frac{\alpha}{2}\right)^2\right)_-,
    \end{equation}
    where the minimum is taken over all eigenvalues of $A_{N_1}$, and $a_-:=\min\{0,a\}$, for any $a\in \mathbb{R}$.
\end{theorem}

\section{Propositions of Bakry-\'Emery curvature}
As a preparation, we present the following Lemma, where the proof is inspired by a recent approach of \cite{HPS24}.

\begin{lemma}\label{lemma1}
    Let $G=(V,E)$ be a graph satisfying $CD(\infty,K)$ at $x\in V$. Suppose that $X,\overline{X}$ partition $N_1(x)$ and $A,\overline{A}$ partition $N_2(x)$. Then, for any constant $\varepsilon\in \mathbb{R}$, we have
    \begin{align}\notag
        &(1-\varepsilon)^2\left[ e(X,\overline{X})+ e(X,\overline{A})+ e(\overline{X},A) -\sum_{z\in A} \frac{\left( d_{\overline{X}}(z) \right)^2}{d_{N_1}(z)} -\sum_{z\in \overline{A}} \frac{d_{X}(z)^2}{d_{N_1}(z)} \right] 
        \\ \notag \ge\ &\frac{1}{4}(2K+d(x)-3)\left( \varepsilon^2|X| +|\overline{X}|\right)+\frac{1}{4} \left( \varepsilon^2 e(X,N_2) + e(\overline{X},N_2) \right) -\frac{1}{2}\left( \varepsilon|X|+ |\overline{X}| \right)^2.
    \end{align}
\end{lemma}

\begin{proof}
    To use Proposition \ref{PH}, we let $f:V\rightarrow \mathbb{R}$ be the function defined as follows:
    \begin{center}
    $f(z)=\begin{cases}
       \varepsilon, &{\rm if}\ z\in X;\\
       1, &{\rm if}\ z\in \overline{X};\\
       2\left( \varepsilon d_X(z)+d_{\overline{X}}(z) \right)/d_{N_1}(z), &{\rm if}\ z\in N_2;\\
       0, &{\rm otherwise}.
    \end{cases}$
    \end{center}
    It follows that
    \begin{equation}\label{3.1}
        \sum_{\substack{yy'\in E\\ y,y'\in N_1}} (f(y)-f(y'))^2= (1-\varepsilon)^2 e(X,\overline{X}).
    \end{equation}
    By the definition of $\Delta$ and $\Gamma$, we have
    \begin{equation}\label{3.2}
        \Delta f(x)= \varepsilon|X|+ |\overline{X}|
    \end{equation}
    and
    \begin{equation}\label{3.3}
        \Gamma (f)(x)= \frac{1}{2}\left( \varepsilon^2|X| +|\overline{X}| \right).
    \end{equation}
    
    For any vertex $z\in N_2$, we have
    \begin{align}\notag
        &\sum_{\substack{y\in N_1 \\ y\sim z}}\left[ \frac{1}{4}(f(z)-f(y))^2-\frac{1}{2} (f(z)-f(y))(f(y)-f(x))\right]
        \\ \notag =\ &d_X(z)\left[ \frac{1}{4}(f(z)-\varepsilon)^2-\frac{1}{2} (f(z)-\varepsilon)\varepsilon\right]
        + d_{\overline{X}}(z)\left[ \frac{1}{4}(f(z)-1)^2-\frac{1}{2} (f(z)-1)\right]
        \\ \notag =\ &\frac{1}{4}d_{N_1}(z) f(z)^2- \left( \varepsilon d_X(z)+d_{\overline{X}}(z) \right) f(z) +\frac{3}{4} \left( \varepsilon^2 d_X(z) + d_{\overline{X}}(z) \right)
        \\ \notag =\ &-\frac{\left( \varepsilon d_X(z)+d_{\overline{X}}(z) \right)^2}{d_{N_1}(z)} + \frac{3}{4} \left( \varepsilon^2 d_X(z) + d_{\overline{X}}(z) \right).
    \end{align}
    Summing the above equation for all $z\in N_2$ yields
    \begin{align}\notag
        &\sum_{z\in N_2}\sum_{\substack{y\in N_1 \\ y\sim z}}\left[ \frac{1}{4}(f(z)-f(y))^2-\frac{1}{2} (f(z)-f(y))(f(y)-f(x)) \right]
        \\ \label{3.4} =\ &-\sum_{z\in N_2} \frac{\left( \varepsilon d_X(z)+d_{\overline{X}}(z) \right)^2}{d_{N_1}(z)} + \frac{3}{4} \left( \varepsilon^2 e(X,N_2) + e(\overline{X},N_2) \right).
    \end{align}
    
    Observe that
    \begin{align}\notag
        \sum_{z\in N_2} \frac{\left( \varepsilon d_X(z)+d_{\overline{X}}(z) \right)^2}{d_{N_1}(z)}&= \sum_{z\in A} \frac{\left( \varepsilon d_X(z)+d_{\overline{X}}(z) \right)^2}{d_{N_1}(z)} + \sum_{z\in \overline{A}} \frac{\left( \varepsilon d_X(z)+d_{\overline{X}}(z) \right)^2}{d_{N_1}(z)}
        \\ \notag &=\sum_{z\in A} \frac{\left( \varepsilon d_{N_1}(z)+(1-\varepsilon) d_{\overline{X}}(z) \right)^2}{d_{N_1}(z)} + \sum_{z\in \overline{A}} \frac{\left( d_{N_1}(z)-(1-\varepsilon)d_{X}(z) \right)^2}{d_{N_1}(z)}.
    \end{align}
    We write the first term of the above expression as follows:
    \begin{align}\notag
        \sum_{z\in A} \frac{\left( \varepsilon d_{N_1}(z)+(1-\varepsilon) d_{\overline{X}}(z) \right)^2}{d_{N_1}(z)} &= \sum_{z\in A} \left[ \varepsilon^2 d_{N_1}(z) +2\varepsilon(1-\varepsilon)d_{\overline{X}}(z) \right] + (1-\varepsilon)^2 \sum_{z\in A} \frac{\left( d_{\overline{X}}(z) \right)^2}{d_{N_1}(z)} 
        \\ \notag &=\varepsilon^2 e(N_1,A) +2\varepsilon(1-\varepsilon) e(\overline{X},A) + (1-\varepsilon)^2 \sum_{z\in A} \frac{\left( d_{\overline{X}}(z) \right)^2}{d_{N_1}(z)}
        \\ \notag &=\varepsilon^2 e(X,A) -\left(\varepsilon^2 -2\varepsilon\right) e(\overline{X},A) + (1-\varepsilon)^2 \sum_{z\in A} \frac{\left( d_{\overline{X}}(z) \right)^2}{d_{N_1}(z)}.
    \end{align}
    Similarly, we have
    \begin{align}\notag
        \sum_{z\in \overline{A}} \frac{\left( d_{N_1}(z)-(1-\varepsilon)d_{X}(z) \right)^2}{d_{N_1}(z)}&= \sum_{z\in \overline{A}} \left[ d_{N_1}(z) -2(1-\varepsilon)d_{X}(z) \right] + (1-\varepsilon)^2  \sum_{z\in \overline{A}} \frac{\left( d_{X}(z) \right)^2}{d_{N_1}(z)}
        \\ \notag &= e(N_1,\overline{A})-2(1-\varepsilon) e(X,\overline{A}) + (1-\varepsilon)^2  \sum_{z\in \overline{A}} \frac{\left( d_{X}(z) \right)^2}{d_{N_1}(z)}
        \\ \notag &= e(\overline{X},\overline{A})-(1-2\varepsilon) e(X,\overline{A}) + (1-\varepsilon)^2  \sum_{z\in \overline{A}} \frac{\left( d_{X}(z) \right)^2}{d_{N_1}(z)}.
    \end{align}
    Then, we deduce that
    \begin{align}\notag
        &\sum_{z\in N_2} \frac{\left( \varepsilon d_X(z)+d_{\overline{X}}(z) \right)^2}{d_{N_1}(z)}
        \\ \notag =\ &\varepsilon^2 e(X,N_2)+e(\overline{X},N_2)- (1-\varepsilon)^2 \left[ e(\overline{X},A)+e(X,\overline{A}) -\sum_{z\in A} \frac{\left( d_{\overline{X}}(z) \right)^2}{d_{N_1}(z)} - \sum_{z\in \overline{A}} \frac{\left( d_{X}(z) \right)^2}{d_{N_1}(z)} \right].
    \end{align}
    Combining the above equation with the equation \eqref{3.4}, we arrive at
    \begin{align}\notag
        &\sum_{z\in N_2}\sum_{\substack{y\in N_1 \\ y\sim z}}\left[ \frac{1}{4}(f(z)-f(y))^2-\frac{1}{2} (f(z)-f(y))(f(y)-f(x)) \right]
        \\ \notag =\ &(1-\varepsilon)^2 \left[ e(\overline{X},A)+e(X,\overline{A}) -\sum_{z\in A} \frac{\left( d_{\overline{X}}(z) \right)^2}{d_{N_1}(z)} - \sum_{z\in \overline{A}} \frac{\left( d_{X}(z) \right)^2}{d_{N_1}(z)} \right] -\frac{1}{4} \left( \varepsilon^2 e(X,N_2) + e(\overline{X},N_2) \right).
    \end{align}
    The desired result then follows by substituting the above equation and equations \eqref{3.1}, \eqref{3.2}, and \eqref{3.3} into Proposition \ref{PH}.
\end{proof}

In the amply regular case, we have the following result.

\begin{corollary}\label{lemma2}
    Let $G$ be an amply regular graph with parameters $(d,\alpha,\beta)$ and $x$ be a vertex in $G$. Suppose that $X,\overline{X}$ partition $N_1(x)$ and $A,\overline{A}$ partition $N_2(x)$. If $G$ satisfies $CD(\infty,K)$ at $x$, then
    \begin{equation}\notag 
        e(X,\overline{X})+ e(X,\overline{A})+ e(\overline{X},A)\ge \frac{1}{4d} \left(2K+2d-\alpha-4\right)
        |X||\overline{X}|.
    \end{equation}
\end{corollary}

\begin{proof}
    If $|X|=0$, the proof is done. Assume that $|X|\ge 1$.
    For any vertex $y\in N_1$, since $x$ and $y$ have exactly $\alpha$ common neighbors, we have $d_{N_1}(y)=\alpha$ and hence $d_{N_2}(y)=d-\alpha-1$.
    It follows that $e(X,N_2)=|X|(d-\alpha-1)$ and $e(\overline{X},N_2)=|\overline{X}|(d-\alpha-1)$.
    Thus, Lemma \ref{lemma1} implies
    \begin{align}\notag
        &(1-\varepsilon)^2\left[ e(X,\overline{X})+ e(X,\overline{A})+ e(\overline{X},A) \right] 
        \\ \notag \ge\ &\frac{1}{4}(2K+d-3)\left( \varepsilon^2|X| +|\overline{X}|\right)+\frac{1}{4} (d-\alpha-1) \left( \varepsilon^2 |X| + |\overline{X}| \right) -\frac{1}{2}\left( \varepsilon|X|+ |\overline{X}| \right)^2.
    \end{align}
    Taking $\varepsilon =-|\overline{X}|/|X|$ yields
    \begin{align}\notag
        \left(\frac{d}{|X|}\right)^2\left[ e(X,\overline{X})+ e(X,\overline{A})+ e(\overline{X},A) \right] \ge\frac{1}{4}(2K+d-3) \frac{d|\overline{X}|}{|X|} +\frac{1}{4} (d-\alpha-1) \frac{d|\overline{X}|}{|X|} .
    \end{align}
    The desired result then follows by dividing both sides by $\left( d/|X| \right)^2$.
\end{proof}

\begin{remark}
    A $d$-regular graph with $n$ vertices is called \emph{edge-regular} with parameters $(n,d,\alpha)$ if any two adjacent vertices have precisely $\alpha$ common neighbors \cite[Section 1.1]{BCN89}. From the proof, we know that  Corollary \ref{lemma2} holds for any edge-regular graph with parameters $(n,d,\alpha)$. Indeed, we also allow infinite edge-regular graphs.  
\end{remark}

\section{Edge-connectivity and Bakry-\'Emery curvature}

\begin{proof}[Proof of Theorem \ref{BE}]
    Let $G=(V,E)$ be a connected graph satisfying $CD(\infty,0)$. Let $\delta$ and $\kappa'$ be the  minimum degree and the edge-connectivity of $G$, respectively. For a contradiction, we assume that $\kappa'\le \delta-2$.
    Then, $\delta\ge \kappa'+2\ge 3$.
    By the definition of edge-connectivity, we suppose that $L,R$ is a partition of $V$ such that $e(L,R)=\kappa'$. For any vertex $z$, let $d_{L,R}(z)$ be $d_{L}(z)$ when $z\in R$ and be $d_{R}(z)$ when $z\in L$. 
    Set 
    $$\Delta_{L,R}:=\max_{z\in V}\{ d_{L,R}(z) \}.$$
    Without loss of generality, assume that $x\in L$ is a vertex with $d_R(x)=\Delta_{L,R}$. 
    For simplicity, we write $d:=d(x)$.
    Set $X:=N_1(x)\cap L$, $\overline{X}:=N_1(x)\cap R$, $A:=N_2(x)\cap L$, and $\overline{A}:=N_2(x)\cap R$. 

    We first show the following claim.
    
    \begin{claim}\label{claim}
        We have $e(\overline{X}, N_2)\ge |\overline{X}|(\delta -2|\overline{X}|+1)$.
    \end{claim}
    \begin{proof}
         By the selection of $x$, we have $d_{L}(y)\le d_{R}(x)=|\overline{X}|$ for any $y\in \overline{X}$. Thus,
    $$d_{N_2}(y)\ge d(y)-d_{L}(y)-d_{\overline{X}}(y)\ge \delta-|\overline{X}|-(|\overline{X}|-1)=\delta -2|\overline{X}|+1.$$
    Summing over all $y\in \overline{X}$ yields $e(\overline{X}, N_2)\ge |\overline{X}|(\delta -2|\overline{X}|+1)$.
    \end{proof}
    
    Since $G$ satisfies $CD(\infty,0)$, Lemma \ref{lemma1} implies particularly that the following inequality holds for any constant $\varepsilon$:
    \begin{align}\notag
        &(1-\varepsilon)^2\left[ e(X,\overline{X})+ e(X,\overline{A})+ e(\overline{X},A)  \right] 
        \\ \label{eq} \ge\ &\frac{1}{4}(d-3)\left( \varepsilon^2|X| +|\overline{X}|\right)+\frac{1}{4}  e(\overline{X},N_2) -\frac{1}{2}\left( \varepsilon|X|+ |\overline{X}| \right)^2.
    \end{align}
    Since \begin{equation}\label{BarX}
        |\overline{X}|=e(\{x\},\overline{X})\le e(L,R) = \kappa'\le \delta-2, 
    \end{equation}
    we have $|X|=d-|\overline{X}|\ne 0$.
    Now, taking $\varepsilon=-|\overline{X}|/|X|$ in \eqref{eq}, we derive
    \begin{align}\label{eq'}
        \left(1+\frac{|\overline{X}|}{|X|}\right)^2\left[ e(X,\overline{X})+ e(X,\overline{A})+ e(\overline{X},A)  \right] 
        \ge\ \frac{1}{4}(d-3)\left( \frac{|\overline{X}|^2}{|X|} +|\overline{X}|\right)+\frac{1}{4}  e(\overline{X},N_2).
    \end{align}
    Note that $$e(X,\overline{X})+ e(X,\overline{A})+ e(\overline{X},A) \le e(L,R)- e(\{x\},\overline{X})= \kappa'- |\overline{X}|\le \delta -2 - |\overline{X}|.$$
    The inequality \eqref{eq'} yields
    \begin{align}\notag
        \left(1+\frac{|\overline{X}|}{|X|}\right)^2\left( \delta -2- |\overline{X}| \right)
        \ge\ \frac{1}{4}(d-3)\left( \frac{|\overline{X}|^2}{|X|} +|\overline{X}|\right)+\frac{1}{4}  e(\overline{X},N_2).
    \end{align}
    Upon multiplying both sides by $(|X|/d)^2$ and replacing $|X|$ by $d-|\overline{X}|$, we arrive at
    \begin{align}\notag
        \delta -2- |\overline{X}| &\ge \frac{d-3}{4d}|\overline{X}|(d-|\overline{X}|) + \frac{1}{4}\left( \frac{d-|\overline{X}|}{d} \right)^2 e(\overline{X},N_2)
        \\ \label{eq3.8} &\ge \frac{\delta-3}{4\delta}|\overline{X}|(\delta-|\overline{X}|) + \frac{1}{4}\left( \frac{\delta-|\overline{X}|}{\delta} \right)^2 e(\overline{X},N_2).
    \end{align}
    
    Now, we divide the proof into four cases according to the size of $\overline{X}$.

    \noindent\textbf{Case 1.} We have $|\overline{X}|\ge 5$.

    The inequality \eqref{eq3.8} implies particularly that
    \begin{equation}\label{eq3.8'}
        \delta -2- |\overline{X}|\ge \frac{\delta-3}{4\delta}|\overline{X}|(\delta-|\overline{X}|).
    \end{equation}
    Consider the above inequality as a quadratic inequality with respect to $|\overline{X}|$. By \eqref{BarX}, we have $5\le |\overline{X}|\le \delta -2$. 
    We will obtain a contradiction by showing that \eqref{eq3.8'} does not hold when $|\overline{X}|$ equals $5$ or $\delta-2$. It is clear that \eqref{eq3.8'} does not hold for $|\overline{X}|=\delta-2$. Now, replacing $|\overline{X}|$ by 5 in \eqref{eq3.8'}, we find
    $4\delta(\delta-7)\ge 5(\delta-3)(\delta-5)$, which is impossible.

    \noindent\textbf{Case 2.} We have $|\overline{X}|=3$ or $|\overline{X}|=4$.

    If $|\overline{X}|=3$, then Claim \ref{claim} gives $e(\overline{X}, N_2)\ge 3(\delta -5)$.
    It follows by \eqref{eq3.8} that \begin{align}\notag
        \delta -5\ge \frac{3(\delta-3)^2}{4\delta} + \frac{3}{4}\left( \frac{\delta-3}{\delta} \right)^2 (\delta -5).
    \end{align}
    Elementary calculation shows that $\delta< 2$, which is a contradiction.

    If $|\overline{X}|=4$, then Claim \ref{claim} gives $e(\overline{X}, N_2)\ge 4(\delta -7)$.
    Similarly, inequality \eqref{eq3.8} yields \begin{align}\notag
        \delta -6\ge \frac{\delta-3}{\delta}(\delta-4) + \left( \frac{\delta-4}{\delta} \right)^2 (\delta -7),
    \end{align}
    which is also impossible for any $\delta\ge 3$.

    To prove the case of $|\overline{X}|\le 2$, inequality \eqref{eq} is not sufficient. We need the following inequality, which is also derived by Lemma \ref{lemma1} but stronger than \eqref{eq}:
    \begin{align}\notag
        &(1-\varepsilon)^2\left[ e(X,\overline{X})+ e(X,\overline{A})+ e(\overline{X},A) -\sum_{z\in \overline{A}} \frac{d_{X}(z)^2}{d_{N_1}(z)} \right] 
        \\ \label{seq} \ge\ &\frac{1}{4}(d-3)\left( \varepsilon^2|X| +|\overline{X}|\right)+\frac{1}{4}  e(\overline{X},N_2) -\frac{1}{2}\left( \varepsilon|X|+ |\overline{X}| \right)^2.
    \end{align}

    \noindent\textbf{Case 3.} We have $|\overline{X}|=1$.

    Set $\overline{X}:=\{y\}$. By the selection of $x$, we know that $x$ is the only neighbor of $y$ in $L$. 
    Hence, we have 
    \begin{equation}\label{eq6}
        e(X,\overline{X})=e(\overline{X},A)=0.
    \end{equation}
    In addition, Claim \ref{claim} shows that
    \begin{equation}\label{eq7}
        e(\overline{X},N_2)\ge \delta-1.
    \end{equation}
    For a vertex $z$ in $\overline{A}$, since $d_{X}(z)$ equals 0 or 1, we have $d_{X}(z)^2=d_{X}(z)$
    and $d_{N_1}(z)= d_{X}(z)+d_{\overline{X}}(z)\le 2$. It follows that
    \begin{equation}\label{eq8}
        \sum_{z\in \overline{A}} \frac{d_{X}(z)^2}{d_{N_1}(z)}=\sum_{z\in \overline{A}} \frac{d_{X}(z)}{d_{N_1}(z)}\ge \sum_{z\in \overline{A}} \frac{d_{X}(z)}{2}= \frac{1}{2}e(X,\overline{A}).
    \end{equation}
    Substituting \eqref{eq6}, \eqref{eq7}, and \eqref{eq8} into \eqref{seq} yields
    \begin{align}\label{eq9}
        \frac{1}{2}(1-\varepsilon)^2 e(X,\overline{A}) \ge\ \frac{1}{4}(d-3)\left( \varepsilon^2(d-1) +1\right)+\frac{1}{4}  (\delta-1) -\frac{1}{2}\left( \varepsilon(d-1)+ 1 \right)^2.
    \end{align}

    By the assumption that $\kappa'\le \delta-2$, we have
    \begin{equation}\label{eq10}
        e(X,\overline{A})\le e(L,R)-e(\{x\},R) =\kappa'-1\le \delta-3.
    \end{equation}
    If $d\ge \delta+1$, let us take $\varepsilon=0$. Then, \eqref{eq9} gives 
    $$e(X,\overline{A})\ge \frac{1}{2}(d-3)+\frac{1}{2}(\delta-1)-1\ge \delta-\frac{5}{2},$$
    which is contradictory to \eqref{eq10}. 
    Now, assume that $d=\delta$. Substituting \eqref{eq10} into \eqref{eq9}, we deduce that
    \begin{align}\label{eq11}
        \frac{1}{2}(1-\varepsilon)^2 (\delta-3) \ge\ \frac{1}{4}(\delta-3)\left( \varepsilon^2(\delta-1) +1\right)+\frac{1}{4}  (\delta-1) -\frac{1}{2}\left( \varepsilon(\delta-1)+ 1 \right)^2.
    \end{align}
    Elementary calculation shows that \eqref{eq11} does not hold for 
    $$-\frac{8}{\delta^2+2\delta-7}< \varepsilon<0,$$
    which is a contradiction. Now, we complete the case of $|\overline{X}|=1$.

    \noindent\textbf{Case 4.} We have $|\overline{X}|=2$.

    Let $\overline{X}:=\{y_1,y_2\}$ and let $c:=e(L,\overline{X})-2$. Note that $c=e(X,\overline{X})+ e(\overline{X},A)$. We have the following estimate for $e(\overline{X},N_2)$:
    \begin{align}\notag
        e(\overline{X},N_2)&=\sum_{i=1}^2 d_{N_2}(y_i) =\sum_{i=1}^2 \left( d(y_i)-1-d_X(y_i)- d_{\overline{X}}(y_i)\right)
        \\ \label{eq12} &\ge \sum_{i=1}^2 \left( \delta-1-d_X(y_i)- 1 \right)\ge 2\delta-4-c.
    \end{align}
    By the selection of $x$, any vertex in $\overline{A}$ has at most two neighbors in $X$.
    Let $a$ be the number of vertices in $\overline{A}$ with exactly one neighbor in $X$, and let $b$ be the number of vertices in $\overline{A}$ with exactly two neighbors in $X$. 
    Then, $e(X,\overline{A})=a+2b$. For any vertex $z\in \overline{A}$, since $|\overline{X}|=2$, we have $d_{N_1}(z)\le d_{X}(z)+2$. It follows that
    \begin{align}\label{eq13}
        \sum_{z\in \overline{A}} \frac{d_{X}(z)^2}{d_{N_1}(z)} \ge \sum_{z\in \overline{A}} \frac{d_{X}(z)^2}{d_{X}(z)+2}=\frac{a}{3}+b\ge \frac{1}{3} e(X,\overline{A}).
    \end{align}
    Substituting \eqref{eq12}, \eqref{eq13}, and the fact that $c=e(X,\overline{X})+ e(\overline{X},A)$ into \eqref{seq}, we derive
    \begin{align}\notag
        (1-\varepsilon)^2\left[  \frac{2}{3} e(X,\overline{A})+ c  \right] \ge\ \frac{1}{4}(d-3)\left( \varepsilon^2(d-2) +2\right)+\frac{1}{4}  (2\delta-4-c) -\frac{1}{2}\left( \varepsilon(d-2)+ 2 \right)^2.
    \end{align}
    Note that $e(X,\overline{A})\le e(L,R)-c-2= \kappa'-c-2\le \delta-c-4$. We have
    \begin{align}\notag
        \frac{1}{3}(1-\varepsilon)^2\left( 2\delta+ c-8  \right) \ge\ \frac{1}{4}(d-3)\left( \varepsilon^2(d-2) +2\right)+\frac{1}{4}  (2\delta-4-c) -\frac{1}{2}\left( \varepsilon(d-2)+ 2 \right)^2.
    \end{align}
    Taking $\varepsilon=-1/(d-2)$ yields
    \begin{align}\label{seq1}
        \frac{1}{3}\left(\frac{d-1}{d-2}\right)^2\left( 2\delta+ c-8  \right) \ge\ \frac{(d-3)(2d-3)}{4(d-2)}+\frac{1}{4}  (2\delta-4-c) -\frac{1}{2},
    \end{align}
    which implies that
    \begin{align}\label{seq2}
        \frac{1}{3}\left(\frac{\delta-1}{\delta-2} \right)^2\left( 2\delta+ c-8  \right) \ge\ \frac{(\delta-3)(2\delta-3)}{4(\delta-2)}+\frac{1}{4}  (2\delta-4-c) -\frac{1}{2}.
    \end{align}
    
    By the selection of $x$, we find
    $c=d_{L}(y_1)+d_{L}(y_2)-2\le 2+2-2=2$. However, elementary calculation shows that \eqref{seq2} does not hold when $c$ equals 0 or 1. Thus, we have $c=2$. Then, \eqref{seq2} gives $\delta\le 7$. Since $\delta\ge \kappa'+2=e(L,R)+2\ge c+4=6$, we know that $\delta$ equals 6 or 7.

    Since $c=2$, both $y_1$ and $y_2$ have exactly two neighbors in $L$. Let $z_1$ and $z_2$ be the neighbors of $y_1$ and $y_2$ in $L\backslash \{x\}$, respectively. We next show that $z_1=z_2$. Otherwise, assume $z_1\ne z_2$. Note that $y_1$ is a vertex in $R$ with $d_L(y_1)=\Delta_{L,R}$. By the same method that we use to prove $c=2$, we may prove that $e(R,\{ x,z_1 \})-2=2$. That is, $d_{R}(z_1)=2$. Similarly, we have $d_{R}(z_2)=2$. It follows that $\delta\ge \kappa'+2=e(L,R)+2\ge d_{R}(x)+d_{R}(z_1)+d_{R}(z_2) +2=8$, which is contradictory to $\delta\le 7$. Thus, we have $z_1=z_2$. For simplicity, set $z_0:=z_1=z_2$.

    Now, we divide case 4 into two subcases according to the value of $\delta$.

    \noindent\textbf{Subcase 1.} We have $\delta =6$.
    
    Since $e(L,R)=\kappa'\le \delta-2=4$, there are exactly four edges between $L$ and $R$, namely $xy_1,xy_2,z_0y_1$, and $z_0y_2$. Then, $e(X,\overline{A})=0$. Substituting $\delta=6$ and $c=2$ into \eqref{seq1} yields $d\le 6$, which implies $d=6$. By symmetry, we have $d(z_0)=d(y_1)=d(y_2)=6$. If $x$ is not adjacent to $z_0$, it follows by \eqref{eq12} that
        $$e(\overline{X},N_2) \ge \sum_{i=1}^2 \left( \delta-1-d_X(y_i)- 1 \right)=2(\delta-2)=8.$$
    By \eqref{seq}, we derive
    $$2(1-\varepsilon)^2\ge \frac{3}{4}(4\varepsilon^2+2) +2-\frac{1}{2}(4\varepsilon+2)^2,$$
    which is not true for $\varepsilon =-2/7$. Now, assume that $x$ is adjacent to $z_0$. Similarly, we may assume that $y_1$ is adjacent to $y_2$.
    
    Let $g: V\to R$ be the function defined as follows:
    \begin{center}
    $g(z)=\begin{cases}
       0, &{\rm if}\ z\in \{ x,z_0 \};\\
       1, &{\rm if}\ z\in \overline{X};\\
       2, &{\rm if}\ z\in R\backslash \overline{X};\\
       -1/2, &{\rm if}\ z\in L\backslash \{ x,z_0 \}.
    \end{cases}$
    \end{center}
    Applying Proposition \ref{PH} with $K=0$ and $d(x)=6$, we have
    \begin{align}\notag
        \sum_{y\in N_1}\sum_{\substack{z\in N_2 \\ z\sim y}}\left[ \frac{1}{4}(g(z)-g(y))^2-\frac{1}{2} (g(z)-g(y))(g(y)-g(x)) \right] +\sum_{\substack{yy'\in E\\ y,y'\in N_1}} (g(y)-g(y'))^2
        \\ \notag \ge  \frac{3}{2} \Gamma (g)(x) -\frac{1}{2}(\Delta g(x))^2.
    \end{align}
    Note that we have interchanged the order of summation. Elementary calculation gives
    $$\Gamma(g)(x)=\frac{11}{8},\ \Delta g(x)=\frac{1}{2},\ \sum_{\substack{yy'\in E\\ y,y'\in N_1}} (g(y)-g(y'))^2=\frac{1}{4}d_X (z_0)+2,$$
    and
    $$\sum_{y\in N_1}\sum_{\substack{z\in N_2 \\ z\sim y}}\left[ \frac{1}{4}(g(z)-g(y))^2-\frac{1}{2} (g(z)-g(y))(g(y)-g(x)) \right] =\frac{1}{16}d_A (z_0)-\frac{3}{2}.$$
    Therefore, we arrive at
    $$\frac{1}{16}d_A (z_0)-\frac{3}{2}+\frac{1}{4}d_X (z_0)+2\ge \frac{33}{16}-\frac{1}{8}.$$
    Noting that $d_A (z_0)+d_X (z_0)=3$, we see that the above inequality cannot hold.

    \noindent\textbf{Subcase 2.} We have $\delta =7$.

    Substituting $\delta=7$ and $c=2$ into \eqref{seq1} yields $d\le 7$, which implies $d=7$. By symmetry, we have $d(z_0)=d(y_1)=d(y_2)=7$. Using \eqref{eq12}, we have $e(\overline{X},N_2) \ge 8$. If $e(X,\overline{A})=0$, it follows by \eqref{seq} that
    $$2(1-\varepsilon)^2\ge (5\varepsilon^2+2)+2
    -\frac{1}{2}(5\varepsilon+2)^2,$$
    which dose not hold when $\varepsilon=-1/2.$ Now, assume that $e(X,\overline{A})> 0$. However, since $e(L,R)\le \delta-2=5$, we have $e(X,\overline{A})=1$. Let $w_1w_2$ be the edge with $w_1\in X$ and $w_2\in \overline{A}$. 
    There are exactly five edges between $L$ and $R$, namely $xy_1,xy_2,z_0y_1,z_0y_2$, and $w_1w_2$.
    By the same method that we use to prove $e(X,\overline{A})=1$, we know that $e(N_1(z_0)\cap L, N_2(z_0)\cap R)=1$. This implies that $z_0$ is adjacent to $w_1$. Similarly, we deduce that both $y_1$ and $y_2$ are adjacent to $w_2$.

    Set $W:=N_1(w_1)\cap L$, $\overline{W}:=N_1(w_1)\cap R$, $B:=N_2(w_1)\cap L$, and $\overline{B}:=N_2(w_1)\cap R$. According to Lemma \ref{lemma1}, we derive
    \begin{align}\notag
        &(1-\varepsilon)^2\left[ e(W,\overline{W})+ e(W,\overline{B})+ e(\overline{W},B) -\sum_{z\in \overline{B}} \frac{d_{W}(z)^2}{d_{N_1(w_1)}(z)} \right] 
        \\ \notag \ge\ &\frac{1}{4}(d(w_1)-3)\left( \varepsilon^2|W| +|\overline{W}|\right)+\frac{1}{4} e\left(\overline{W},N_2(w_1)\right) -\frac{1}{2}\left( \varepsilon|W|+ |\overline{W}| \right)^2.
    \end{align}
    Note that $e(W,\overline{W})=e(\overline{W},B)=0$, $e(W,\overline{B})=4$,  $e\left(\overline{W},N_2(w_1)\right)\ge \delta -1=6$, and
    $$\sum_{z\in \overline{B}} \frac{d_{W}(z)^2}{d_{N_1(w_1)}(z)} =\sum_{i=1}^2 \frac{d_{W}(y_i)^2}{d_{N_1(w_1)}(y_i)}=2\times \frac{2^2}{3}=\frac{8}{3}.$$
    We conclude that
    \begin{align}\notag
        \frac{4}{3}(1-\varepsilon)^2 \ge\frac{1}{4}(d(w_1)-3)\left( \varepsilon^2 \left( d(w_1)-1\right) +1 \right)+\frac{3}{2} -\frac{1}{2}\left( \varepsilon\left( d(w_1)-1\right)+ 1 \right)^2.
    \end{align}
    Taking $\varepsilon=0$ leads to
    \begin{equation*}
        \frac{4}{3} \ge\frac{1}{4}(d(w_1)-3)+\frac{3}{2} -\frac{1}{2},
    \end{equation*}
    which contradicts $d(w_1)\ge \delta=7$. We complete  Subcase 2 and establish Theorem \ref{BE}.
\end{proof}

\section{Edge-connectivity of amply regular graphs}

\begin{proof}[Proof of Theorem \ref{ARG}]
    Let $G=(V,E)$ be an amply regular graph with parameters $(n,d,\alpha,\beta)$ such that $\beta\ge 2$.
    If $d=2$, then $G$ is a quadrangle, and our proof is done. Now, assume that $d\ge 3$.
    Let $\kappa'$ be the edge-connectivity of $G$, and let $L,R$ be a partition of $V$ with $e(L,R)=\kappa'$.
    For any vertex $z$, let $d_{L,R}(z)$ be $d_{L}(z)$ when $z\in R$ and be $d_{R}(z)$ when $z\in L$. Let $$\Delta_{L,R}:=\max\{ d_{L,R}(z) |z\in V \}.$$
    Note that our aim is to show that $\Delta_{L,R}=d$. For a contradiction, we assume that $\Delta_{L,R}\le d-1$.

    Our first goal is to show the following claim. Recall that we always assume $\beta\geq 2$.

    \begin{claim}\label{claim2}
        If $\beta\ge \min\{3,\alpha\}$, then for any vertex $z$, we have $d_{L,R}(z)=0$ or $d_{L,R}(z)\ge \beta$. 
    \end{claim}

    \begin{proof}
        Let $z_1$ be a vertex with $d_{L,R}(z_1)\ge 1$ such that, for any vertex $z$ with $d_{L,R}(z)\ge 1$, we have $d_{L,R}(z_1)\le d_{L,R}(z)$.
        We assume for the sake of contradiction that $d_{L,R}(z_1)\le \beta -1$. 
        Without loss of generality, assume that $z_1\in L$. Let $z_2$ be a neighbor of $z_1$ in $R$. 

        Since $d_L(z_2)\le \Delta_{L,R}\le d-1$, we have $N_1(z_2)\cap R \ne \emptyset$.
        For any vertex $z\in N_1(z_2)\cap R$, if $d_L(z)=0$, then $z_1$ and $z$ have exactly $\beta$ common neighbors and all the $\beta$ common neighbors lie in $R$, which contradicts the assumption that $d_{R}(z_1)\le \beta -1$. 
        Thus, we have $d_L(z)\ge 1$ for any $z\in N_1(z_2)\cap R$. 
        If there is a vertex $z_0\in N_1(z_2)\cap R$ such that $d_L(z_0)\ge 2$, then
        \begin{align}\notag
            e(L,R) &\ge d_L(z_2)+d_L(z_0)+ \sum_{z\in (N_1(z_2)\cap R)\backslash \{z_0\}}d_L(z)
            \\ \notag &\ge d_L(z_2)+2+(d_{R}(z_2)-1)=d+1,
        \end{align}
        which is contradictory to $e(L,R)=\kappa'\le d$. Thus, we have $d_L(z)= 1$ for any $z\in N_1(z_2)\cap R$. By the selection of $z_1$, we also have $d_R(z_1)= 1$. Since
        $$d\ge e(L,R)\ge d_L(z_2)+ \sum_{z\in N_1(z_2)\cap R}d_L(z)=d_L(z_2)+d_R(z_2)=d,$$
        we deduce that $\{z_2\}\cup (N_1(z_2)\cap R)$ are all the vertices in $R$ with at least one neighbor in $L$.

        Fix a vertex $z_3\in N_1(z_2)\cap R$. Let $z_4$ be the neighbor of $z_3$ in $L$. Note that $z_1$ and $z_3$ have at most two common neighbors, namely $z_2$ and $z_4$. It follows that $\beta =2$ and $z_1$ is adjacent to $z_4$. By the assumption of this Claim, we have $\alpha\le 2$.
        If $z_3$ has a neighbor $z_5$ in $R$ such that $z_5$ has no neighbors in $L$, then $z_4$ and $z_5$ has $\beta =2$ common neighbors. 
        Let $z_6$ be the common neighbor of $z_4$ and $z_5$, distinct from $z_3$. 
        Since $z_6$ has a neighbor $z_4$ in $L$, we have $z_6\in \{z_2\}\cup (N_1(z_2)\cap R)$. Since $z_5$ has no neighbors in $L$, we have $z_6\ne z_2$, and hence $z_6\in N_1(z_2)\cap R$. Then, $z_3$ and $z_6$ have at least 3 common neighbors ($z_2$,$z_4$, and $z_5$), which contradicts $2=\beta\ge\alpha$. Therefore, each neighbor of $z_3$ in $R$ has at least one neighbor in $L$. It follows that $N_1(z_3)\subset \{z_4,z_2\}\cup ((N_1(z_2)\cap R)\backslash \{z_3\})$. Then,
        \begin{equation*}
            d=|N_1(z_3)|\le |\{z_4,z_2\}\cup ((N_1(z_2)\cap R)\backslash \{z_3\})|=2+(d-d_L(z_2)-1)\le d,
        \end{equation*}
        which implies that $d_L(z_2)=1$ and $z_3$ is adjacent to every vertex in $(N_1(z_2)\cap R)\backslash \{z_3\}$. Since $|(N_1(z_2)\cap R)\backslash \{z_3\}|=d-2\ge 1$, we can take a vertex $z_7\in (N_1(z_2)\cap R)\backslash \{z_3\}$. Then $z_7$ is a common neighbor of $z_2$ and $z_3$, which implies $\alpha\ge 1$. However, since $d_R(z_1)=d_L(z_2)=1$, $z_1$ and $z_2$ have no common neighbors, which is a contradiction.
    \end{proof}

    Without loss of generality, assume that there is a vertex $x\in L$ with $d_R(x)=\Delta_{L,R}$. Set $X:=N_1(x)\cap L$ and $\overline{X}:=N_1(x)\cap R$.
    Let us divide $X$ into two parts $X_1$ and $X_2$, where $X_1:=\{z\in X| d_R(z)=0\}$ and $X_2:=\{z\in X| d_R(z)\ge 1\}$. Next, we prove the following claims.

    \begin{claim}\label{claim3}
        For any vertex $z$, we have $d_{L,R}(z)=0$ or $d_{L,R}(z)\ge \alpha +2-|\overline{X}|$. 
    \end{claim}

    \begin{proof}
        Let $z$ be a vertex with $d_{L,R}(z)\ge 1$. If $z\in L$, let $z'$ be a neighbor of $z$ in $R$. For any $w\in N_1(z)\cap N_1(z')$, let $e_w$ be $zw$ when $w\in R$ and be $z'w$ when $w\in L$. Then, we find
        $$d_R(z)+d_L(z')\ge |\{e_w|w\in N_1(z)\cap N_1(z')\}|+2=\alpha+2,$$
        which implies $d_R(z)\ge \alpha +2-d_L(z')$. By the selection of $x$, we have $d_L(z')\le |\overline{X}|$. It follows that $d_{R}(z)\ge \alpha +2-|\overline{X}|$. Similarly, if $z\in R$, then we have $d_{L}(z)\ge \alpha +2-|\overline{X}|$.
    \end{proof}

    \begin{claim}\label{claim4}
        We have $|\overline{X}|\ge \left\lceil \frac{\alpha}{2} \right\rceil+1$. 
    \end{claim}

    \begin{proof}
        According to Claim \ref{claim3}, we have $d_{R}(x)\ge \alpha +2-|\overline{X}|$. That is, $|\overline{X}|\ge  \frac{\alpha}{2} +1$.
    \end{proof}

    \begin{claim}\label{claim5}
        For any vertex $y\in \overline{X}$, we have $$d_{L}(y)\ge \frac{|X_1|(\beta -1)}{\max\{\alpha,\beta\}-1}+1.$$
    \end{claim}

    \begin{proof}
        Fix a vertex $y\in \overline{X}$. Let $M$ be the number of pairs $(w,z)$ such that $w\in X_1$ and $z\in (N_1(y)\cap N_1(w))\backslash \{x\}$. On one hand, for any $w\in X_1$, $w$ and $y$ have exactly $\beta -1$ common neighbors in addition to $x$. Thus,
        \begin{equation}\label{5.1}
            M=|X_1|(\beta -1).
        \end{equation}
        On the other hand, for any $z\in N_1(y)\backslash \{x\}$, let us consider the number of neighbors of $z$ in $X_1$. If $z\in R$, by the definition of $X_1$, $z$ has no neighbors in $X_1$. If $z\in L$, since $x$ and $z$ have at most $\max\{\alpha,\beta\}-1$ common neighbors in addition to $y$, $z$ has at most $\max\{\alpha,\beta\}-1$ neighbors in $X_1$. It follows that
        \begin{equation}\label{5.2}
            M\le (d_{L}(y)-1)(\max\{\alpha,\beta\}-1).
        \end{equation}
        The desired result then follows by combining \eqref{5.1} and \eqref{5.2}.
    \end{proof}

    Now, we are prepared to prove Theorem \ref{ARG}. Let us divide our proof into 3 cases.

\noindent\textbf{Case 1.} We have $\alpha\ge\beta\ge 3$.

    According to Claim \ref{claim5}, we have
    \begin{equation}\label{5.3}
        d\ge e(L,R)\ge \sum_{y\in \overline{X}}d_{L}(y) \ge |\overline{X}|\left( \frac{|X_1|(\beta -1)}{\alpha-1}+1 \right)\ge |\overline{X}|\left( \frac{2|X_1|}{\alpha-1}+1 \right).
    \end{equation}
    Moreover, by Claim \ref{claim2}, we arrive at 
    \begin{equation*}
        d\ge e(L,R)\ge d_R(x)+ \sum_{z\in X_2}d_R(z)\ge |\overline{X}|+\beta|X_2|\ge |\overline{X}|+3|X_2| =3d-2|\overline{X}|-3|X_1|,
    \end{equation*}
    which implies $|X_1|\ge 2(d-|\overline{X}|)/3$. Substituting it into \eqref{5.3} yields
    \begin{equation*}
        d\ge |\overline{X}|\left( \frac{4(d-|\overline{X}|)}{3(\alpha-1)}+1 \right).
    \end{equation*}
    Recall that $|\overline{X}|=\Delta_{L,R} \le d-1$. Upon dividing by $d-|\overline{X}|$, we obtain
    \begin{equation}\label{5.4}
        |\overline{X}|\le \frac{3(\alpha -1)}{4}.
    \end{equation}
    Applying Claim \ref{claim3}, we have
    \begin{align}\notag
        d\ge e(L,R)\ge d_R(x)+ \sum_{z\in X_2}d_R(z)&\ge |\overline{X}|+(\alpha+2-|\overline{X}|)|X_2| 
        \\ \label{5.5} &=|\overline{X}|+(\alpha+2-|\overline{X}|)(d-|\overline{X}|-|X_1|).
    \end{align}
    Note that \eqref{5.4} implies $\alpha +2-|\overline{X}|>0$. Then, \eqref{5.5} yields
    $$|X_1|\ge (d-|\overline{X}|)\left( 1-\frac{1}{\alpha +2-|\overline{X}|} \right).$$
    Substituting it into \eqref{5.3}, we derive
    \begin{equation*}
        d\ge |\overline{X}|\left( \frac{2(d-|\overline{X}|)}{\alpha-1}\left( 1-\frac{1}{\alpha +2-|\overline{X}|} \right)+1 \right).
    \end{equation*}
    By dividing by $d-|\overline{X}|$, we have
    $$1\ge \frac{2|\overline{X}|}{\alpha-1}\left( 1-\frac{1}{\alpha +2-|\overline{X}|} \right).$$
    Then, elementary calculus shows that we have either 
    $|\overline{X}|\ge (3\alpha+1+\sqrt{\alpha^2-2\alpha+17})/4$ or $|\overline{X}|\le (3\alpha+1-\sqrt{\alpha^2-2\alpha+17})/4$. However, the first inequality contradicts \eqref{5.4}, and the second inequality contradicts Claim \ref{claim4}.

    \noindent\textbf{Case 2.} We have $\beta\ge\alpha$.

    According to Claim \ref{claim5}, we have
    \begin{equation}\label{5.6}
        d\ge e(L,R)\ge \sum_{y\in \overline{X}}d_{L}(y) \ge |\overline{X}|\left( |X_1|+1 \right).
    \end{equation}
    Furthermore, by Claim \ref{claim2}, we deduce that 
    \begin{equation}\label{5.7}
        d\ge e(L,R)\ge d_R(x)+ \sum_{z\in X_2}d_R(z)\ge |\overline{X}|+\beta|X_2|\ge |\overline{X}|+2|X_2| =2d-|\overline{X}|-2|X_1|,
    \end{equation}
    which implies
    \begin{equation}\label{5.8}
        |X_1|\ge \frac{d-|\overline{X}|}{2}.
    \end{equation}
    Substituting this into \eqref{5.6} yields
    \begin{equation*}
        d\ge |\overline{X}|\left( \frac{d-|\overline{X}|}{2}+1 \right).
    \end{equation*}
    That is, $|\overline{X}|\le 2$. However, Claim \ref{claim2} implies $|\overline{X}|\ge 2$. It follows that $|\overline{X}|= 2$, and all inequalities \eqref{5.6}, \eqref{5.7}, and \eqref{5.8} take equal. Thus, we have 
    $|X_1|= (d-2)/2$, and hence $|X_2|= (d-2)/2$. 
    Equality \eqref{5.7} implies that $\beta =2$, and $d_R(z)=2$ for any $z\in X_2$. Set $\overline{X}:=\{y_1,y_2\}$. Equality \eqref{5.6} gives $e(L,R)= \sum_{i=1}^2d_{L}(y_i)$. Therefore, each vertex in $X_2$ is a common neighbor of $y_1$ and $y_2$. However, since $\alpha\le\beta =2$, $y_1$ and $y_2$ have at most one common neighbor, except $x$. It follows that $|X_2|\le 1$, which implies $d=4$. Set $X_1:=\{ z_1 \}$ and $X_2:=\{ z_2 \}$. Then $xy_1,xy_2,z_2y_1$, and $z_2y_2$ are all the edges between $L$ and $R$. Note that $|N_1(z_1)\backslash \{ x,z_2\}|\ge d-2=2$. There exists a vertex $z_3\in N_1(z_1)\backslash \{ x,z_2\}$ such that $z_3$ is not adjacent to $z_2$. Since $N_1(x)=\{z_1,z_2,y_1,y_2\}$, $z_3$ is not adjacent to $x$. However, $z_3$ and $x$ have only one common neighbor $z_1$, which is contradictory to $\beta =2$.

    \noindent\textbf{Case 3.} We have $\alpha>\beta=2$.

    Applying Theorem \ref{thm:ARGCurvature} with $\beta =2$, we find
    \begin{equation}\label{Kx}
    K_{BE}(x)=2+\frac{\alpha}{2}+\left(-\frac{\alpha^2}{4}+\min_{\lambda\in \mathrm{sp} \left(A_{N_1}\right)}\left(\lambda-\frac{\alpha}{2}\right)^2\right)_-,
    \end{equation}
    where $A_{N_1}$ is the adjacency matrix of the subgraph induced by $N_1(x)$. Let us set $A:=N_2(x)\cap L$ and $\overline{A}:=N_2(x)\cap R$. 
    Then, Corollary \ref{lemma2} shows that
    \begin{equation}\notag 
        e(X,\overline{X})+ e(X,\overline{A})+ e(\overline{X},A)\ge \frac{1}{4d} \left(2K_{BE}(x)+2d-\alpha-4\right)
        |X||\overline{X}|.
    \end{equation}
    Since $e(X,\overline{X})+ e(X,\overline{A})+ e(\overline{X},A)\le e(L,R)-e(\{x\},\overline{X})\le d-|\overline{X}|=|X|$, we derive
    \begin{equation}\label{5.10}
        1\ge \frac{1}{4d} \left(2K_{BE}(x)+2d-\alpha-4\right)|\overline{X}|.
    \end{equation}

    Now, let us divide Case 3 into two subcases.

    \noindent\textbf{Subcase 1.} We have $d\ge \frac{1}{2}\alpha(\alpha+ 3)$.

    Note that \eqref{Kx} implies $K_{BE}(x)\ge 2+\frac{\alpha}{2}-\frac{\alpha^2}{4}$. Then, inequality \eqref{5.10} yields
    \begin{equation*}
        1\ge \frac{1}{4d} \left(2\left( 2+\frac{\alpha}{2}-\frac{\alpha^2}{4} \right)+2d-\alpha-4\right)|\overline{X}|.
    \end{equation*}
    It follows that
    \begin{equation*}
        |\overline{X}|\le \frac{4d}{2d-\frac{\alpha^2}{2}} = 2+ \frac{\alpha^2}{2d-\frac{\alpha^2}{2}} \le 2+\frac{2\alpha}{\alpha +6}.
    \end{equation*}
    By Claim \ref{claim4}, we have $|\overline{X}|\ge \left\lceil \frac{\alpha}{2} \right\rceil+1$. Then, elementary calculation gives $\alpha\le 2$, which contradicts the assumption in Case 3.

    \noindent\textbf{Subcase 2.} We have $d< \frac{1}{2}\alpha(\alpha+ 3)$.

    By Theorem \ref{BCN}, the subgraph $G[N_1]$ induced by $N_1(x)$ is a disjoint union of some cliques. Since $G[N_1]$ is $\alpha$-regular, it is a disjoint union of some cliques of size $\alpha+1$. Thus, the adjacency matrix $A_{N_1}$ of $G[N_1]$ satisfies
    $$\mathrm{sp} \left( A_{N_1} \right)=\{ -1,\alpha \}.$$
    Then, \eqref{Kx} yields $K_{BE}(x)=2+\frac{\alpha}{2}$. It follows by \eqref{5.10} that $|\overline{X}|\le 2$. However, since $\alpha\ge 3$, Claim \ref{claim4} implies $|\overline{X}|\ge 3$, which is a contradiction.
\end{proof}
    
\section*{Acknowledgement}
This work is supported by the National Key R \& D Program of China 2023YFA1010200 and the National Natural Science Foundation of China No. 12431004. K. C. is partially supported by the New Lotus Scholars Program PB22000259. J. H. K. is partially supported by the National Natural Science Foundation of China No. 12471335.

\end{document}